\documentclass[twoside, 10pt]{article}
\usepackage{amssymb, amsmath, mathrsfs, amsthm}
\usepackage{graphicx}
\usepackage{color}
\usepackage[ps2pdf=true]{hyperref}
\usepackage[top=2cm, bottom=2cm, left=2cm, right=2cm]{geometry}
\usepackage{float, caption, subcaption}

\newcommand{\bd}{\begin{displaymath}}
\newcommand{\ed}{\end{displaymath}}
\newcommand{\bi}{\begin{itemize}}
\newcommand{\ei}{\end{itemize}}
\newcommand{\be}{\begin{enumerate}}
\newcommand{\ee}{\end{enumerate}}
\newcommand{\beqs}{\begin{eqnarray*}}
\newcommand{\eeqs}{\end{eqnarray*}}

\definecolor{DarkGreen}{rgb}{0.2, 0.6, 0.3}

\newcommand{\kr}{\epsilon}
\renewcommand{\r}{r}
\renewcommand{\d}{d}
\newcommand{\C}{C}

\newtheorem{theorem}{Theorem}
\newtheorem{lemma}{Lemma}
\newtheorem{definition}{Definition}

\newtheorem{case}{Case}

\newtheorem{claim}{Claim}

\newtheorem{fact}{Fact}

\newtheorem{conjecture}{Conjecture}
\theoremstyle{definition}
\newtheorem{remark}{Remark}

\begin{document}
\title{\bf Enomoto and Ota's conjecture holds for large graphs}
\author{Vincent Coll\footnote{Department of Mathematics, Lehigh University, Bethlehem, PA, USA. {\tt vec208@lehigh.edu}}, Alexander Halperin\footnote{Department of Mathematics and Computer Science, Salisbury University, Salisbury, MD, USA. {\tt adhalperin@salisbury.edu}}, Colton Magnant\footnote{Department of Mathematical Sciences, Georgia Southern University, Statesboro, GA, USA.  {\tt cmagnant@georgiasouthern.edu}}, Pouria Salehi-Nowbandegani\footnote{Department of Mathematical Sciences, Georgia Southern University, Statesboro, GA, USA. {\tt pouria\underline{~}salehi-nowbandegani@georgiasouthern.edu}}}
\maketitle

\begin{abstract}
In 2000, Enomoto and Ota conjectured that if a graph $G$ satisfies $\sigma_{2}(G) \geq n + k - 1$, then for any set of $k$ vertices $v_{1}, \dots, v_{k}$ and for any positive integers $n_{1}, \dots, n_{k}$ with $\sum n_{i} = |G|$, there exists a partition of $V(G)$ into $k$ paths $P_{1}, \dots, P_{k}$ such that $v_{i}$ is an end of $P_{i}$ and $|P_{i}| = n_{i}$ for all $i$.  We prove this conjecture when $|G|$ is large.  Our proof uses the Regularity Lemma along with several extremal lemmas, concluding with an absorbing argument to retrieve misbehaving vertices.
\end{abstract}

\section{Introduction}

For all basic definitions and notation, see~\cite{CLZ11}. Let $\sigma_2(G)$ denote the minimum degree sum of a graph $G$. In 2000, Enomoto and Ota conjectured the following and proved several cases.

\begin{conjecture}[Enomoto and Ota \cite{EO00}]\label{Conj:E-O}
Given an integer $k \geq 3$, let $G$ be a graph of order $n$ and let $n_{1}, n_{2}, \dots, n_{k}$ be a set of $k$ positive integers with $\sum n_{i} = n$.  If $\sigma_{2}(G) \geq n + k - 1$, then for any $k$ distinct vertices $x_{1}, x_{2}, \dots, x_{k}$ in $G$, there exists a set of vertex-disjoint paths $P_{1}, P_{2}, \dots, P_{k}$ such that $|P_{i}| = n_{i}$ and $P_{i}$ starts at $x_{i}$ for all $i$ with $1 \leq i \leq k$.
\end{conjecture}

A partial solution was provided by Magnant and Martin in the sense that the path lengths could only be prescribed within a small fraction of $n$.

\begin{theorem}[Magnant and Martin \cite{MM10}]
Given an integer $k \geq 3$, for every set of $k$ positive real numbers $\eta_{1}, \dots \eta_{k}$ with $\sum_{i = 1}^{k} \eta_{i} = 1$, and for every $\epsilon > 0$, there exists $n_{0}$ such that for every graph $G$ of order $n \geq n_{0}$ with $\sigma_{2}(G) \geq n + k - 1$ and for every choice of $k$ vertices $S = \{ x_{1}, \dots, x_{k} \} \subseteq V(G)$, there exists a set of vertex disjoint paths $P_{1}, \dots, P_{k}$ which span $V(G)$ with $P_{i}$ beginning at the vertex $x_{i}$ and $(\eta_{i} - \epsilon)n < |P_{i}| < (\eta_{i} + \epsilon)n$.  Also the condition on $\sigma_{2}(G)$ is sharp.
\end{theorem}

When $n$ is sufficiently large relative to $k$, we prove that Conjecture~\ref{Conj:E-O} holds.

\begin{theorem}\label{Thm:Main}
Given an integer $k \geq 3$, let $G$ be a graph of sufficiently large order $n$ and let $n_{1}, n_{2}, \dots, n_{k}$ be a set of $k$ positive integers with $\sum n_{i} = n$.  If $\sigma_{2}(G) \geq n + k - 1$, then for any $k$ distinct vertices $x_{1}, x_{2}, \dots, x_{k}$ in $G$, there exists a set of vertex disjoint paths $P_{1}, P_{2}, \dots, P_{k}$ such that $|P_{i}| = n_{i}$ and $P_{i}$ starts at $x_{i}$ for all $i$ with $1 \leq i \leq k$.
\end{theorem}

Our proof utilizes several extremal lemmas based on the structure of the reduced graph provided by the Regularity Lemma.  Our lemmas deal with the cases where the minimum degree is small, the reduced graph has a large independent set and the connectivity of the reduced graph is small.

\section{Preliminaries}

Given two sets of vertices $A$ and $B$, let $E(A, B)$ denote set of edges with one end in $A$ and one end in $B$ and let $e(A, B) =
|E(A, B)|$.  Define the \emph{density between $A$ and $B$} to be
\begin{equation*}
d(A, B) = \frac{e(A, B)}{|A||B|}.
\end{equation*}

\begin{definition}
Let $\epsilon > 0$.  Given a graph $G$ and two nonempty disjoint vertex sets $A, B \subset V$, we say that the pair $(A, B)$ is \emph{$\epsilon$-regular} if for every $X \subset A$ and $Y \subset B$ satisfying
\begin{equation*}
|X| > \epsilon |A| ~ ~ ~ \text{and} ~ ~ ~ |Y| > \epsilon |B|
\end{equation*}
we have
\begin{equation*}
|d(X, Y) - d(A, B)| < \epsilon.
\end{equation*}
\end{definition}

We will also use the following one-sided but stronger version of regularity.

\begin{definition}
Let $\epsilon, \d > 0$.  Given a graph $G$ and two nonempty disjoint vertex sets $A, B \subset V$, we say that the pair $(A,B)$ is \emph{$(\epsilon, \d)$-super-regular} if for every $X \subset A$ and $Y \subset B$ satisfying
\begin{equation*}
|X| > \epsilon |A| ~ ~ ~ \text{and} ~ ~ ~ |Y| > \epsilon |B|,
\end{equation*}
we have
\begin{equation*}
e(X, Y) > \d|X| |Y|,
\end{equation*}
and furthermore $d_{B}(a) > \d|B|$ for all $a \in A$ and $d_{A}(b) > \d|A|$ for all $b \in B$.
\end{definition}

The following is the famous Regularity Lemma of Szemer{\'e}di.

\begin{lemma}[Regularity Lemma - Szemer{\'e}di \cite{S78}]\label{Lemma:Regularity}
For every $\epsilon > 0$ and every positive integer $m$, there is an $M = M(\epsilon)$ such that if $G$ is any graph and $\d \in(0,1)$ is any real number, then there is a partition of $V(G)$ into $\r+1$ clusters $V_{0},V_{1},\dots,V_{\r}$, and there is a subgraph $G^\prime \subseteq G$ with the following properties:
\begin{itemize}
\item[(1)] $m \leq \r \leq M$,\label{Item:Clusters} 
\item[(2)] $|V_{0}|\leq \epsilon|V(G)|$,\label{Item:Garbage size} 
\item[(3)] $|V_{1}| = \dots = |V_{\r}|=L \leq
\epsilon|V(G)|$,\label{Item:Cluster size} 
\item[(4)] $\deg_{G^\prime}(v)>\deg_{G}(v)-(\d+\epsilon)|V(G)|$ for all $v\in
V(G)$,\label{Item:Degree size} 
\item[(5)] $e(G^\prime[V_i]) = 0$ for all $i \geq 1$,\label{Item:Partite cluster} 
\item[(6)] for all $1\leq i < j\leq \r$ the graph $G^\prime[V_{i},V_{j}]$ is $\epsilon$-regular and has density either $0$ or greater than $\d$.\label{Item:Density d or 0}
\end{itemize}
\end{lemma}

Given a graph $G$ and appropriate choices of $\epsilon$ and $\d$, let $G^\prime$ be a spanning subgraph of $G$ obtained from Lemma~\ref{Lemma:Regularity}. The \emph{reduced graph} $R = R(G,\epsilon,\d)$ of $G$ contains a vertex $v_{i}$ for each cluster $V_{i}$ in $G^\prime\setminus V_0$ and has an edge between $v_{i}$ and $v_{j}$ if and only if $d(V_i,V_j) > \d$. Hence, $V(R) = \{v_i \;|\; 1 \leq i \leq \r\}$ and $E(R) = \{v_iv_j\;|\;1 \leq i,j \leq \r, \; d(V_i,V_j) > \d\}$. 

Throughout this work, we let $\r = |R|$.

This next lemma allows the creation of a super-regular pair from an $\epsilon$-regular pair by simply removing some vertices.

\begin{lemma}[Diestel \cite{D10}, Lemma 7.5.1]
Let $(A, B)$ be an $\epsilon$-regular pair of density $\d$ and let $Y \subseteq B$ have size $|Y| \geq \epsilon |B|$.  Then all but
at most $\epsilon |A|$ of the vertices in $A$ each have at least $(\d - \epsilon)|Y|$ neighbors in $Y$.
\end{lemma}

We use a simple corollary of this result.

\begin{lemma}\label{Lemma:Super-Regular}
Let $(A, B)$ be an $\epsilon$-regular pair of density $\d$.  There exist subsets $A' \subseteq A$ and $B' \subseteq B$ with
$|A'| \geq (1 - \epsilon)|A|$ and $|B| \geq (1 - \epsilon)|B|$ such that the pair $(A', B')$ is $(\epsilon, \d - 2\epsilon)$-super-regular.
\end{lemma}

Our next lemma follows trivially from the definition of super-regular pairs.

\begin{lemma}\label{Lemma:ShortPathInSuper}
Given an $(\epsilon, \d)$-super-regular pair $(A, B)$ and a pair of vertices $a \in A$ and $b \in B$, there exists a path of length at
most $3$ from $a$ to $b$ in $(A, B)$.
\end{lemma}

The following lemma provides very strong structure in super-regular pairs.

\begin{lemma}[Blow-Up Lemma - Koml{\'o}s, S{\'a}rk{\"o}zy and Szemer{\'e}di \cite{KSS97}]\label{Lemma:Blow-Up}
Given a graph $R$ of order $r$ and positive parameters $\d , \Delta ,$ there exists a positive $\epsilon = \epsilon (\d , \Delta , r)$ such that the following holds. Let $\{n_{1},n_{2},\dots , n_{r}\}$ be an arbitrary set of positive integers and replace the set of vertices $\{v_{1},v_{2},\dots ,v_{r}\}$ of $R$ with pairwise disjoint sets $V_{1}, V_{2}, \dots, V_{r - 1}$ and $V_{r}$ of sizes $n_{1},n_{2}, \dots, n_{r - 1}$ and $n_{r}$  respectively (blowing up).  We construct two graphs on the same vertex-set $V = \cup V_{i}$. The first graph ${\bf R}$ is obtained by replacing each edge ${v_i,v_j}$ of $R$ with the complete bipartite graph between the corresponding vertex-sets $V_i$ and $V_j$.  A sparser graph $G$ is constructed by replacing each edge ${v_i,v_j}$ arbitrarily with an $(\epsilon, \d)$-super-regular pair between $V_i$ and $V_j$. If a graph $H$ with $\Delta (H) \leq \Delta$ is embeddable into ${\bf R}$ then it is also embeddable into $G$.
\end{lemma}

The following theorem gives us a degree sum condition on $R$ based on our assumed degree sum condition on $G$.

\begin{theorem}[K{\"u}hn, Osthus and Treglown \cite{KOT09}]\label{Thm:RDegreeSum}
Given a constant $c$, if $\sigma_{2}(G) \geq cn$, then $\sigma_{2}(R) \geq (c - 2\d - 4\epsilon)|R|$.
\end{theorem}

We also use the following theorem of Ore.

\begin{theorem}[Ore \cite{O60}]\label{Thm:Ore}
If $G$ is $2$-connected, then $G$ contains a cycle of length at least $\sigma_{2}(G)$.
\end{theorem}

We use the following result of Williamson.  Recall that a graph is called \emph{panconnected} if, between every pair of vertices, there is a path of every possible length from $2$ up to $n - 1$.

\begin{theorem}\label{Thm:Williamson}
Let $G$ be a graph of order $n$.  If $\delta(G) \geq \frac{n + 2}{2}$, then $G$ is panconnected.
\end{theorem}

\section{Proof Outline}

Given an integer $k \geq 3$ and desired path orders $n_{1}, n_{2}, \dots, n_{k}$, we choose constants $\epsilon$ and $\d$ as follows:
\begin{equation*}
0 < \epsilon \ll \d \ll \frac{1}{k},
\end{equation*}
where $a \ll b$ is used to indicate that $a$ is chosen to be sufficiently small relative to $b$. Let $n$ be sufficiently large to apply Lemma~\ref{Lemma:Regularity} with constant $\epsilon$ to get large clusters and let $R$ be the corresponding reduced graph.  Note that, when applying Lemma~\ref{Lemma:Regularity}, there are at least $\frac{1-\epsilon}{\epsilon}$ clusters so $|R| \geq \frac{1-\epsilon}{\epsilon}$.

We use a sequence of lemmas to eliminate extremal cases of the proof.  Without loss of generality, we assume $n_{1} \leq n_{2} \leq \dots \leq n_{k}$.  Our first lemma establishes the case when $\delta(G)$ is small.

\begin{lemma}\label{Lemma:Degree}
Conjecture~\ref{Conj:E-O} holds when $\delta(G) \leq
\frac{n_{k}}{8}$.
\end{lemma}

Lemma~\ref{Lemma:Degree} is proven in Section~\ref{Section:Degree}.  By Lemma~\ref{Lemma:Degree}, we may assume $\delta(G) \geq \frac{n_{k}}{8} \geq \frac{n}{8k}$.  Our next lemma establishes the case when $\kappa(R) \leq \kr|R|$, where $R$ is the reduced graph of $G$ after applying Lemma~\ref{Lemma:Regularity}.

\begin{lemma}\label{Lemma:Connectivity}
Given a positive integer $k$, let $\epsilon = \epsilon_k, \d = \d_k > 0$, and let $G$ be a graph of order $n \geq n(\epsilon,\d,k)$ with $\sigma_2(G) \geq n+k-1$ and $\delta(G) \geq \frac{n_k}{8}$. If $\kappa(R) \leq 1$, then the conclusion of Conjecture~\ref{Conj:E-O} holds. 
\end{lemma}

Lemma~\ref{Lemma:Connectivity} is proven in Section~\ref{Section:Connectivity}. Our final lemma establishes the case where $G$ contains a large independent set.

\begin{lemma}\label{Lemma:IndepSet}
Given a positive integer $k$, let $\epsilon = \epsilon_k > 0$ be small, and let $G$ be a graph of order $n \geq n(\epsilon)$. If $\sigma_2(G) \geq n + k - 1$ and $\alpha(G) \geq \left(\frac{1}{2} - \epsilon\right)n$, then $G$ satisfies Conjecture~\ref{Conj:E-O}.
\end{lemma}

Lemma~\ref{Lemma:IndepSet} is proven in Section~\ref{Section:IndepSet}.

With all these lemmas in place, we use Ore's Theorem (Theorem~\ref{Thm:Ore}) to construct a long cycle in the reduced graph of $G$. Alternating edges of this cycle are then made into super-regular pairs of $G$.  This structure is then used to construct the desired paths.  The complete proof of our main result, assuming the above lemmas, is presented in the following section.

\section{Proof of Theorem~\ref{Thm:Main}}\label{Section:MainPf}

By Lemma~\ref{Lemma:Connectivity}, we may assume $R$ is $2$-connected.  By Theorem~\ref{Thm:RDegreeSum}, we know that $\sigma_{2}(R) \geq (1 - 2\d - 4\epsilon)|R|$.  Thus, we may apply Theorem~\ref{Thm:Ore} to obtain a cycle $\C$ of length at least $(1 - 2\d - 4\epsilon)|R|$ in $R$.  Define a ``garbage set'' to include $V_{0}$ and those clusters not used in $\C$.

Color the edges of $\C$ with red and blue such that no two red edges are adjacent and as few blue edges as possible are adjacent.  Note that if $\C$ is even, then the colors will alternate, and if $\C$ is odd, then there will be only one consecutive pair of blue edges while all others are alternating.  Apply Lemma~\ref{Lemma:Super-Regular} on the pairs of clusters in $G$ corresponding to the red edges of $R$ to obtain super-regular pairs where the two sets of each super-regular pair have the same order.  All vertices discarded in this process are added to the garbage set and define the clusters $C_{i}$ to be the original clusters without the removed vertices.  Note that we have added a total of at most $\epsilon n$ vertices to the garbage set.

If $C$ is odd, then let $c_{0}$ be the vertex in $R$ with two blue edges, let $C_{0}$ be the corresponding cluster in $G$, and let $C_{0}^{+}$ and $C_{0}^{-}$ be the neighboring clusters in $G$.  Since the pairs $(C_{0}^{-}, C_{0})$ and $(C_{0}, C_{0}^{+})$ are both large and $\epsilon$-regular, there exists a set of at least $k$ vertices $T_{0} \subseteq C_{0}$ with a matching to each of $C_{0}^{-}$ and $C_{0}^{+}$.  We will use these vertices as transportation and move all of $C_{0} \setminus T_{0}$ to the garbage set.

Let $G_{\C}$ denote the graph on the set of vertices remaining in clusters associated with $\C$ that have not been moved to the garbage set, and let $D$ denote the garbage set. Then $V(G) = \bigcup_{i=1}^|\C| C_i \cup C_0 \cup D$ (if $C_0$ exists), with $|D| \leq (2\d + 7\epsilon) n$.

By Lemma~\ref{Lemma:Degree}, we may assume $\delta(G) \geq \frac{n_{k}}{8}$.  In particular, the vertices in $D$ each have at least $\frac{n_{k}}{8} - (|D| - 1) \gg \epsilon n$ edges to $G_{\C}$.

A path is said to \emph{balance} the super-regular pairs in $G_{\C}$ if for every super-regular pair the path visits, it uses an equal number of vertices from each set in the pair.  Note that the removal of a balancing path preserves the fact that if a pair of clusters is super-regular, then the two clusters have the same order.  Let $(A, B)$ be a super-regular pair of clusters in $G_{\C}$. A balancing path starting in $A$ and ending in $B$ which contains a vertex $v \in D$ is called \emph{$v$-absorbing}.

\begin{claim}\label{Claim:Absorbing}
Avoiding any selected set of at most $\epsilon \r$ clusters and any set of at most $\frac{16(2\d + 7\epsilon)n}{\epsilon \r}$ vertices in each of the remaining clusters, there exists a $v$-absorbing path of order at most $17$.  Otherwise the desired path partition already exists.
\end{claim}

\begin{proof}
Absorbing paths are constructed iteratively, one for each vertex of $D$, in an arbitrary order.  Suppose some number of such absorbing paths have been created.  If we have created one for each vertex of $D$ within the restrictions of the claim, the proof is complete so suppose we have constructed at most $|D| - 1$ absorbing paths.  Vertices that have already been used and clusters that have lost at least $\frac{16(2\d + 7\epsilon)n}{\epsilon \r}$ vertices removed from consideration in following iterations.

Recall that $L$ is the order of each non-garbage cluster of $G$ in Lemma~\ref{Lemma:Regularity}.

\begin{fact}
If we have created at most $|D| - 1$ such paths, at most $\epsilon \r$ clusters would have order at most $L - \frac{16(2\d + 7\epsilon)n}{\epsilon \r}$.
\end{fact}

\begin{proof}
Since each absorbing path constructed in this claim has order at most $16$ (other than the vertex $v$), we lose at most $16$ vertices from $G_{\C}$ for each vertex of $D$.  The result follows.
\end{proof}

Let $v \in D$ such that there is no absorbing path for $v$ of order at most $16$.  Since $d(v) \geq \frac{n_{k}}{8}$, $v$ must have edges to at least $\frac{\r}{8k}$ clusters.  Let $A$ and $B$ be two clusters which are not already ignored to which $v$ has at least two edges to vertices that are not already in a path or an absorbing path.  For convenience, we call two clusters $X$ and $Y$ a \emph{couple} or \emph{spouses} if $X$ and $Y$ are consecutive on $\C$ and the pair $(X, Y)$ is super-regular.

The following facts are easily proven using the structure we have provided and the lemmas proven before.

\begin{fact}
$A$ and $B$ are not a couple.
\end{fact}

Otherwise it would be trivial to produce a $v$-absorbing path.

Let $A'$ and $B'$ denote the spouses of $A$ and $B$, respectively, let $a',b'\in R$ correspond to $A'$ and $B'$, respectively, and define the following sets of clusters:

\bi
\item $X_A:=\{$all couples $(P,Q)$ of clusters such that $pa'$ and $qa'$ are edges in $R\}$,
\item $X_B:=\{$all couples of clusters such that both clusters have an edge to $B'$ in $R\}$, and
\item $X_{AB}:=\{$all couples of clusters such that one spouse has an edge to both $A'$ and $B'$ in $R\}$.  In particular, let $X_{AB}'$ denote the clusters in $X_{AB}$ that are not the neighbors of $A'$ and $B'$.
\ei

Since we are considering two neighbors of $v$ in $A$ (and two neighbors of $v$ in $B$), say $v_{1}$ and $v_{2}$, if $X_{A}$ (or similarly $X_{B}$) contains even a single couple $(Q, R)$, then we can absorb $v$ using a path of the form $v_{1}vv_{2} - A' - Q - R - A'$.  Thus, we may actually assume $X_{A} = X_{B} = \emptyset$.

Our next fact follows from the fact that $\sigma_{2}(R) \geq (1 - 2\d - 4\epsilon)|R|$.

\begin{fact}
There are at most $(2\d - 4\epsilon)|R|$ clusters in $C$ which are not in $X_{AB}$.
\end{fact}

If there is an edge $xy$ between two (non-used) vertices in clusters in $X_{AB}'$, then there is a $v$-absorbing path of the form $v_{1}vv_{2} - A' - (X_{AB} \setminus X_{AB}') - xy - (X_{AB} \setminus X_{AB}') - A'$.  Thus, the graph induced on the vertices in clusters in $X_{AB}'$ contains no edges.  By Lemma~\ref{Lemma:IndepSet}, we have the desired set of paths. This completes the proof of Claim~\ref{Claim:Absorbing}.
\end{proof}

For each chosen vertex $x_{i}$, if $x_{i} \notin G_{C}$, use Menger's Theorem~\cite{M1927} to construct a shortest path to a vertex, say $x_{i}'$, in $G_{C}$.  Using an edge of a super-regular pair first, construct a balancing path from $x_{i}'$ through every cluster of $G_{C}$.  Note that, since the pairs are either $\epsilon$-regular or $(\epsilon, \d)$-super-regular, using Lemma~\ref{Lemma:ShortPathInSuper}, this path can be constructed to use at most $2$ vertices from each cluster.

First suppose the path starting at $x_{i}$ already has order at least $n_{i} - 16$.  In this case, we add at most $16$ vertices using a super-regular pair (and Lemma~\ref{Lemma:ShortPathInSuper}) or discard any excess vertices to obtain the desired path.  If a coupled pair of clusters in $G_{C}$ is left unbalanced by this process, we simply remove a vertex from the larger cluster to $D$.  Note that repeating this for each short path adds at most $k - 1$ vertices to $D$.

By Claim~\ref{Claim:Absorbing}, since $|D| \leq (2\d + 7\epsilon)n$, we can construct an absorbing path for each vertex $v \in D$ where these paths are all disjoint.  Let $P^{v}$ be an absorbing path for $v$ with ends of $P^{v}$ in clusters $C_{i}$ and $C_{i + 1}$.  Suppose
$uw$ is the edge of $P_{k}$ from $C_{i}$ to $C_{i + 1}$.  Then using Lemma~\ref{Lemma:ShortPathInSuper}, we can replace the edge $uw$ with the path $P^{v}$ with the addition of at most $4$ extra vertices at either end.  Note that absorbing a vertex $v \in D$ into a path $P_{i}$ using the absorbing path will always correct the parity of the length of $P_{i}$.

For each path $P_{i}$ that is not already completed and not the correct parity, absorb a single vertex from $D$ into $P_{i}$.  This will correct the parity of the path.

Recall the assumption (without loss of generality) that $n_1 \leq \dots \leq n_k$. By the same process, all remaining vertices of $D$ can be absorbed into $P_{k}$.  This makes $|P_{k}|$ larger but since $|D| \leq (2\d + 7\epsilon)n$ and each absorbing path $P^{v}$ for $v \in D$ has order at most $17$, we get $|P_{k}| \leq 3|C| + 17(2\d + 7\epsilon)n < n_{k}$.

The following lemma, stated in \cite{MS14}, is an easy exercise using the definitions of $(\epsilon, \d)$-super-regular pairs.

\begin{lemma}[Magnant and Salehi \cite{MS14}]\label{Lemma:OnePair} 
Let $U$ and $V$ be two clusters forming a balanced $(\epsilon, \d)$-super-regular pair with $|U| = |V| = L$.  Then for every pair of vertices $u \in U$ and $v \in V$, there exist paths of all odd lengths $\ell$ between $u$ and $v$ satisfying
\begin{description}
\item{\rm (a)} $3 \leq \ell \leq \d L$ and
\item{\rm (b)} $(1 - \d)L \leq \ell \leq L$.
\end{description}
\end{lemma}

For each $i$ with $n_{i}$ small, absorb a few pairs of vertices from each super-regular pair, using Lemma~\ref{Lemma:OnePair}, until $P_{i}$ has the desired order.  For each remaining index $i$, using Lemma~\ref{Lemma:OnePair} absorb entire super-regular pairs at a time (along with possibly a few vertices from other super-regular pairs) until each path $P_{i}$ has the desired order to complete the proof.

\section{Proof of Lemma~\ref{Lemma:Degree}}\label{Section:Degree}

Recall that Lemma~\ref{Lemma:Degree} claims Conjecture~\ref{Conj:E-O} holds when $\delta(G) \leq \frac{n_{k}}{8}$.

\begin{proof}
Let $a \in V(G)$ with $|N(a)|=\delta(G) \leq \frac{n_{k}}{8}$, and partition $V(G)$ as follows:
\beqs
B & = & G \setminus (a \cup N(a))\\
A & = & \left\{ v \in a \cup N(a): |N(v) \cap V(B)| < \frac{1}{8}(n+k-\delta(G)-1) \right\}\\
C & = & \left\{ v \in a \cup N(a): |N(v) \cap V(B)| \geq \frac{1}{8}(n+k-\delta(G)-1) \right\}
\eeqs
Note that, since $\sigma_2(G) \geq n+k-1$, the set $A$ induces a complete graph. Furthermore, the set $B$ has order $n-1-\delta(G)$, and $A$ is nonempty since $a \in A$.  Since $\sigma_{2}(G) \geq n + k - 1$ and $a$ has no edges to $B$, each vertex in $B$ has degree at least $n + k - 1 - \delta(G)$ which means $\delta(G[B]) \geq n + k - 1 - 2\delta(G)$.  Note that also $G$ is at least $(k + 1)$-connected.  First, a claim about subsets of $B$.

\begin{claim}\label{Claim:DegSmallPanConn}
Every subset of $B$ of order at least $\frac{3n_{k}}{8}$ is panconnected.
\end{claim}

\begin{proof}
With $|B| = n - \delta(G) - 1$ and $\delta(G[B]) \geq n + k - 1 - 2\delta(G)$, we see that $\delta(G[B]) \geq |B| - \delta(G) \geq |B| - \frac{n_{k}}{8}$.  Therefore, for any subset $B' \subseteq B$ with $|B'| \geq \frac{3n_{k}}{8}$, we have $\delta(G[B']) \geq |B'| - \frac{n_{k}}{8} > \frac{|B'| + 2}{2}$.  By Theorem~\ref{Thm:Williamson}, we see that $B'$ is panconnected.
\end{proof}

Consider $k$ selected vertices $X = \{x_1, \dots, x_k\} \subseteq V(G)$. Let $X_{A}$ denote the (possibly empty) set $X \cap A$ and let $X_{A}'$ denote $X_{A} \cup v$ where $v \in A \setminus X_{A}$ if such a vertex $v$ exists.  If no such vertex $v$ exists, then let $X_{A}' = X_{A}$.  The vertices of $X_{A}'$ will serve as start vertices for paths that will be used to cover all of $A$.  By Menger's Theorem, since $\kappa(G) \geq k + 1$, there exists a set of disjoint paths $\mathscr{P}_{A}$ starting at the vertices of $X_{A}'$ and ending in $B$ and avoiding all other vertices of $X$. Choose $\mathscr{P}_A$ so that each path is as short as possible, contains only one vertex in $B$ and, by construction, has order at most $4$.  If any of the paths in $\mathscr{P}_{A}$ begins at a selected vertex $x_{i}$ and has order at least $n_{i}$, we call this desired path completed and remove the first $n_{i}$ vertices of the path from the graph and continue the construction process.  If $A \setminus V(\mathscr{P}_{A}) \neq \emptyset$, let $P_{v}$ be a path using all remaining vertices and ending at $v$.  This path $P_{v}$ together with the path of $\mathscr{P}_{A}$ corresponding to $v$ provides a single path that cleans up the remaining vertices of $A$ and ends in $B$.  The ending vertices of these paths, the vertices of $B$, will serve as proxy vertices for the start vertices ($v$ or $x_{i} \in X \cap A$).  Thus far, we have constructed paths that cover all of $A$, start at vertices of $X \cap A$ (when such vertices exist) and end in $B$.

As vertices of $B$ are selected and used on various paths, we continuously call the set of vertices in $B$ that have not already been prescribed or otherwise mentioned the \emph{remaining} vertices in $B$.  For example, so far, we have $B \setminus (X \cup V(\mathscr{P}_{A}))$ are the remaining vertices of $B$.  Our goal is to maintain at least $\frac{3n_{k}}{8}+1$ remaining vertices to be able to apply Claim~\ref{Claim:DegSmallPanConn} as needed within these remaining vertices.

Since $|C| \leq \delta(G) \leq \frac{n_{k}}{8}$ and $d_{B}(u) \geq \frac{1}{8} (n + k - \delta(G) - 1)$ for all $u \in C$, there exists a set of two distinct neighbors in $B \setminus (X \cup V(\mathscr{P}_{A}))$ for each vertex in $C$.  For each vertex $x_{i} \in X \cap C$, select one such vertex to serve as a proxy for $x_{i}$ and leave the other aforementioned neighbor in the remaining vertices of $B$.  By Claim~\ref{Claim:DegSmallPanConn}, there exists a path through the remaining vertices of $B$ with at most one intermediate vertex from one neighbor of a vertex of $C$ to a neighbor of another vertex of $C$.  Since $|C| \leq \frac{n_{k}}{8}$, such paths can be built and strung together into a single path $P_{C}$ starting and ending in $B$, containing all vertices of $C \setminus X$ with $|P_{C}| < 4 |C| \leq \frac{n_{k}}{2}$.

We may now construct what is left of the desired paths within $B$. The paths $P_{1}, P_{2}, \dots, P_{k - 1}$ can be constructed in any order starting at corresponding proxy vertices and ending at arbitrary remaining vertices of $B$ using Claim~\ref{Claim:DegSmallPanConn} in the remaining vertices of $B$.  Finally, there are at least
\beqs
|B| - \left|B \cap \left(\cup_{i = 1}^{k - 1} V(P_{i})\right)\right| - |B \cap V(\mathscr{P}_{A})| - |B \cap V(P_{C})| & \geq & (n - 1 - \delta(G)) - (k + 1) - (3|C|)\\
~ & > & \frac{3n_{k}}{8} + 1
\eeqs remaining vertices in $B$.
With these and Claim~\ref{Claim:DegSmallPanConn}, we construct a path with at most one internal vertex from an end of $P_{C}$ to the proxy of $v$ (if such a vertex exists) and a path containing all remaining vertices of $B$ from $x_{k}$ (or its proxy) to the other end of $P_{C}$.  This completes the construction of the desired paths and thereby completes the proof of Lemma~\ref{Lemma:Degree}.
\end{proof}

\section{Proof of Lemma~\ref{Lemma:IndepSet}}\label{Section:IndepSet}

Recall that Lemma~\ref{Lemma:IndepSet} says for a positive integer $k$, a small $\epsilon = \epsilon_k > 0$, and a graph $G$ of order $n \geq n(\epsilon)$, if $\sigma_2(G) \geq n + k - 1$ and $\alpha(G) \geq \left(\frac{1}{2} - \epsilon\right)n$, then $G$ satisfies Conjecture~\ref{Conj:E-O}.

\begin{proof}
Let $A$ be a maximum independent set of $G$, and let $B = V(G) \setminus A$. By the assumption on $\sigma_2(G)$, we have $|B| \geq \frac{1}{2}(n + k - 1)$. This implies
\begin{equation*}
\left(\frac{1}{2} - \epsilon\right)n \leq |A| \leq \frac{1}{2}(n - k + 1)
\end{equation*}
and
\begin{equation}\label{Equation: B}
\frac{1}{2}(n + k - 1) \leq |B| \leq \left(\frac{1}{2} + \epsilon\right)n.
\end{equation}

Let $t= |B|-|A|+k$; therefore, we have
\begin{equation*}
2k-1 \leq t \leq 2\epsilon n + k.
\end{equation*}
Let $C$ denote the set of vertices in $B$ with fewer than $\frac{t-1}{2}$ neighbors in $B$. Since $\frac{1}{2}(n+k-1) - |A| \leq \frac{|B|-|A|+k-1}{2} = \frac{t-1}{2}$, we see $C$ is necessarily a clique on at most $\frac{t-1}{2}$ vertices. 

\begin{claim}\label{Lemma:B'}
Let $B'$ be the set of vertices in $B$ with at least $\frac{|A|-1}{|B|}\cdot\frac{1}{2}(n+k-1) > \left(\frac{1}{2}-\frac{1}{100k}\right)n$ neighbors in $A$. Then $|B'| \geq \frac{1}{2}(n+k-1)$.
\end{claim}
\begin{proof}
Each vertex in $A$ (except possibly one) has at least $\frac{1}{2}(n+k-1)$ neighbors in $B$, which means there are at least $(|A|-1)\cdot \frac{1}{2}(n+k-1)$ edges between $A$ and $B$. By the Pigeonhole Principle, we have our result.
\end{proof}


A bipartite graph $U\cup V$ is \emph{bipanconnected} if for every pair of vertices $x,y\in U\cup V$, there exist $(x,y)$-paths of all possible lengths at least 2 of appropriate parity in $U\cup V$. That is, for every pair of vertices $x\in U$ and $y\in V$, there exist $(x,y)$-paths of every possible odd length except 1, and for every pair of vertices $x,y\in U$ (and $V$), there exist $(x,y)$-paths of every even length. Note that we must exclude the value 1 from our definition in order to allow graphs $U\cup V$ that are not complete bipartite. Also observe that the partite sets of a bipanconnected graph must have order within one of each other.

\begin{lemma}[Coll, Halperin and Magnant \cite{CHM13}]\label{Lemma:Bipanconnected}
If $G[U \cup V]$ is a balanced bipartite graph of order $2m$ with $\delta(G[U \cup V]) \geq \frac{3m}{4}$, then $G[U \cup V]$ is bipanconnected.
\end{lemma}

Our next claim shows that any reasonably large subsets of $B'$ induce bipanconnected subgraphs when paired with any corresponding subset of $A$.

\begin{claim}
For all $m \geq \frac{n}{25k}$, the $m$-subsets $U\subseteq A$ and $V\subseteq B'$ induce a bipanconnected subgraph of order $2m$.
\end{claim}
\begin{proof}
For $m \geq \frac{n}{25k}$, let $U$ and $V$ be $m$-subsets of $A$ and $B'$, respectively. Each vertex in $U$ has at least $|V| - \epsilon n + \frac{k-1}{2} > \frac{3m}{4}$ neighbors in $V$. Each vertex in $B'$ has greater than $|U| - \left(|A| - \left(\frac{1}{2}-\frac{1}{100k}\right)n\right) > \frac{3m}{4}$ neighbors in $U$. It follows from Lemma~\ref{Lemma:Bipanconnected} that $G[U \cup V]$ is bipanconnected.
\end{proof}

Let $D = B\setminus B'$, and let $D'$ be the set of vertices in $B$ with fewer than $\frac{n}{100}$ neighbors in $A$. Note that every vertex in $D'$ has well over $\frac{n}{3}$ neighbors in $B'$.

\begin{claim}\label{Lemma:Few None in A}
The sets $D$ and $D'$ satisfy $D'\subseteq D$ and
\begin{equation}
|D'| \leq |D| \leq \frac{t-2k+1}{2}. 
\end{equation}
\end{claim}
\begin{proof}
By the definition of $D$, we have $D'\subseteq D$. By Claim~\ref{Lemma:B'}, we observe
\begin{equation*}
\begin{aligned}
|D|	&\leq |B| - \frac{1}{2}(n+k-1) \\
	&= |B| - \frac{n}{2} - \frac{k-1}{2} \\
	&= \frac{|B|-|A|}{2} - \frac{k-1}{2};
\end{aligned}
\end{equation*}
it follows that $|D'| \leq |D| \leq \frac{t-2k+1}{2}$.
\end{proof}

When creating the desired path partition of $G$, we must make sure that all vertices in $G$ are included. In particular, since $|B| > |A|$, we must ensure that each vertex in $B$ may be included in some path. We show this by creating a path $Q$ that contains more vertices in $B$ than in $A$. The path $Q$ will contain all vertices in $D'$ and roughly (but not necessarily exactly) the appropriate number of additional vertices in $B$. We make this discrepancy exact by considering $Q'$, the collection of paths including $Q$ and all vertices in $D$. 

Before we create $Q$ and $Q'$, we first specify the necessary value $\tau$, the number of additional vertices in $B$ that we must account for before creating our desired paths. Let $X = \{x_1,\dots,x_k\}$ be the set of arbitrarily chosen endpoints for the desired paths. Let $O_A$ be the set of vertices $x_i\in A\cap X$ such that $n_i$ is odd. Let $O_B'$ be the set of vertices $x_i\in B'\cap X$ such that $n_i$ is even. Recall that $D'\subseteq D$ is the set of vertices in $B$ with fewer than $2|D|$ neighbors in $A$. Let $E_D'$ be the set of vertices $x_i\in D'\cap X$ such that $n_i$ is even. Lastly, let $P_D'$ be the path containing 
\begin{itemize}
\item all vertices in $D'\setminus X$,
\item 
two unique neighbors $b'_j,\beta'_j\in B'\setminus X$ of $d'_j$, and
\item for $j \geq 1$, the unique neighbor $a_j\in A\setminus X$ of $\beta'_j$ and $b_{j+1}$.
\end{itemize}
To ensure $\tau \geq 0$ (see Claim~\ref{Lemma:Tau Bounds}) or that our desired path $Q$, an extension of $P_D'$, always begins in $B$ and ends in $A$, we may need to drop the first vertex $b'_1$ of $P_D'$ so that $P_D'$ begins at $d'_1$. As a result, the path $P_D'$ contains either $2|D'\setminus X|$ or $2|D'\setminus X| - 1$ more vertices in $B$ than in $A$ and of course only exists if $D'\setminus X$ is nonempty. Thus, we have $P_D' = \left\{(b'_1),d'_1,\beta'_1,a_1,b'_2,\dots,b'_{|D'\setminus X|},d'_{|D'\setminus X|},\beta'_{|D'\setminus X|},a_{|D'\setminus X|}\right\}$, with the vertex $b'_1$ in parentheses to denote that $b'_1$ is not always included in $P_D'$.

\begin{claim}\label{Lemma:Tau}
Let $G=A\cup B$ as defined in Lemma~\ref{Lemma:IndepSet}. Then
\begin{equation}\label{Equation:Tau}
\tau = 
\begin{cases}
   |B| - |A| + |O_A| - |O_B'| - |D'\cap X| - |E_D'| 		           &\text{if } |D'\setminus X| = 0 \\
   |B| - |A| + |O_A| - |O_B'| - |D'\cap X| - |E_D'| - (|P_D'\cap B| - |P_D'\cap A|)		   &\text{if } |D'\setminus X| \geq 1.
  \end{cases}
\end{equation}
\end{claim}
\begin{proof}
We prove this statement by showing that each term in~\eqref{Equation:Tau} accounts for a necessary quantity. Since $|B| > |A|$ and $A$ is an independent set, it is clear that we must account for $|B|-|A|$ additional vertices in $B$. For each $x_i\in O_A$ (symmetrically $x_i\in O_B'$), the path beginning at $x_i$ that alternates between $A$ and $B$ uses one additional vertex in $A$ (symmetrically $B'$). This accounts for the term $|O_A| - |O_B'|$. Next, we cannot assume $e(D',A) > 0$ and hence must use at least one additional vertex in $B$ to account for each vertex in $D'$. By Claim~\ref{Lemma:B'}, every vertex in $D'$ contains sufficiently many neighbors in $B'\setminus X$ and hence, two unique neighbors in $B'\setminus X$. For all $x_i\in D'\cap X$, reserve the unique vertex $b_i\in B'\setminus X$ and its unique neighbor $a_i\in A\setminus X$. Hence, for each $x_i\in D'\cap X$, we have used one additional vertex in $B$, which accounts for the term $-|D'\cap X|$. Now, if $x_i\in E_D'$, then the desired path starting at $x_i$ and alternating between $A$ and $B$ will contain an \emph{extra} additional vertex in $B$ since the first two vertices and the last vertex of the path are all in $B$. This accounts for the term $-|E_D'|$. Lastly, the only way we can be guaranteed to include the vertices in $D'\setminus X=\left\{d'_1,\dots,d'_{|D'\setminus X|}\right\}$ in a path is to create $P_D'$ containing either $2|D'\setminus X|$ or $2|D'\setminus X| - 1$ more vertices in $B$ than in $A$.

Hence, we have~\eqref{Equation:Tau}.
\end{proof}

\begin{claim}\label{Lemma:Tau Bounds}
$0 \leq t - 2k \leq \tau \leq t - 2|D'\cap X| \leq t$.
\end{claim}
\begin{proof}
First note that $-k \leq |O_A|-|O_B'| \leq k$, and that $|D'| \leq |D| \leq \frac{|B|-|A|-k+1}{2}$. Also, we have $|D'\cap X| + |O_B'| \leq k$ and $|D'\cap X| + |O_A| \leq k$. Similarly, we have $|E_D'| + |O_B'| \leq k$ and $|E_D'| + |O_A| \leq k$. 
It is clear that a situation where $|O_A| = k$ and $|O_B| = |D| = 0$ results in a maximum value for $\tau$. It follows from~\eqref{Equation:Tau} that we have
\begin{equation}\label{Equation:Tau Max}
\begin{aligned}
\tau &\leq |B| - |A| + (k - |D'\cap X|) - 0 - |D'\cap X| - 0 \\
     &\leq t  - 2|D'\cap X| \\
     &\leq t.
\end{aligned}
\end{equation}

It also is clear that a situation where $|O_A| = 0$, $E_D' = D'\cap X$, and $|D'\cap X| + |O_B'| = k$ results in a minimum value for $\tau$. By~\eqref{Equation:Tau}, this gives
\begin{equation*}
\tau \geq 
\begin{cases}
|B| - |A| + 0 - (k-|D'\cap X|) - |D'\cap X| - |D'\cap X|		       		 &\text{if } |D'\setminus X| = 0 \\
|B| - |A| + 0 - (k-|D'\cap X|) - |D'\cap X| - |D'\cap X| - (|P_D'\cap B| - |P_D'\cap A|) &\text{if } |D'\setminus X| \geq 1,
\end{cases}
\end{equation*}
which reduces to
\begin{equation}\label{Equation:Tau Min}
\tau \geq 
\begin{cases}
t - 2k - |D'\cap X|      				&\text{if } |D'\setminus X| = 0 \\
t - 2k - |D'\cap X| - (|P_D'\cap B| - |P_D'\cap A|)     &\text{if } |D'\setminus X| \geq 1.
\end{cases}
\end{equation}

First consider the case where $|D'\setminus X| = 0$. Recall from Lemma~\ref{Lemma:Few None in A} that $|D'| \leq |D| \leq \frac{t-2k+1}{2}$. Then
\begin{equation*}
\begin{aligned}
t - 2k - |D'\cap X| &\geq t - 2k - |D'| \\
		    &\geq t - 2k - \frac{t-2k+1}{2} \\
		    &= \frac{t-2k-1}{2}.
\end{aligned}
\end{equation*}
Hence, if $t - 2k \geq 1$, then $\tau \geq 0$. However, we can still have $t - 2k = 0$ or $t - 2k = -1$. If $t - 2k = 0$, then $|A| = \frac{1}{2}(n-k)$ and $|B| = \frac{1}{2}(n+k)$, which implies $|D| = |D'\cap X| = 0$, and hence, that $\tau \geq t - 2k - |D'\cap X| = 0$. If $t - 2k = -1$, then $|A| = \frac{1}{2}(n-k+1)$ and $|B| = \frac{1}{2}(n+k-1)$, which implies $|D| = 0$. It follows that $n+k-1$ is even, and that either $n$ is odd and $k$ is even, or that $n$ is even and $k$ is odd. Either way, we must have $|O_B| \leq k-1$ or else the parities of $n$ and $k$ are violated. Since $|O_B| = k-|D'\cap X|$, we must have $|D'\cap X| \geq 1$, which contradicts $|D| = 0$. Hence, we have $\tau \geq 0$.

Now consider the case where $|D'\setminus X| \geq 1$. Again recalling Lemma~\ref{Lemma:Few None in A}, we have
\begin{equation*}
\begin{aligned}
t - 2k - |D'\cap X| - (|P_D'\cap B| - |P_D'\cap A|) &\geq t - 2k - 2|D'|.
\end{aligned}
\end{equation*}
It follows that $t - 2k - |D'\cap X| - (|P_D'\cap B| - |P_D'\cap A|)$ if and only if $|P_D'\cap B| - |P_D'\cap A| = 2|D'\setminus X|$ and $|D'\cap X| = 0$. Letting $|P_D'\cap B| - |P_D'\cap A| = 2|D'\setminus X| - 1$ ensures $\tau \geq 0$.
%
\end{proof}

The fact that $\tau \geq 0$ ensures that, no matter what, the set $B$ contains enough vertices to account for a (relatively) large $D'$ and/or any number of odd-ordered paths beginning in $B$.


\begin{claim}\label{Lemma:Easy Indep}
If $G$ is a graph of sufficiently large order $n$ with $\delta(G) \geq \omega$, then for all $\gamma > 0$, the graph $G$ either contains
\begin{enumerate}
\item $\gamma \omega$ independent edges, or
\item $\lceil \omega\rceil$ vertices of degree at least $\frac{n}{2\gamma+1}$.
\end{enumerate}
\end{claim}
\begin{proof}
Consider a graph $G$ of sufficiently large order $n$ with $\delta(G) \geq \omega$. If there exist $\gamma\omega$ independent edges in $G$, then we are done, so suppose not. Consider the largest collection of $c<\gamma\omega$ independent edges, i.e., the largest matching in $G$. Call this set of edges $M_c$. Then $G = M_c \cup A$, where $A$ must induce an independent set of $n-2c$ vertices. Since $\delta(G) \geq \omega$, each vertex in $A$ must be adjacent to $\omega$ vertices in $M_c$. This means there are at least $\omega(n-2c)$ edges from $A$ to $M_c$. Therefore, there are at least $\omega$ vertices in $M_c$ with degree at least $\frac{\omega n-2c\omega}{2c} = \left(\frac{\omega}{c}\right)\frac{n}{2}-\omega > \frac{n}{2\gamma+1}$.
\end{proof}

Recall from before that there may exist a small clique $C\subset B$ of vertices in $B$ with fewer than $\frac{t-1}{2}$ neighbors in $B$. If $C\neq \emptyset$, then we can only be guaranteed that $\delta(G[B]) \geq |C|-1$. For each vertex $c\in C$, attach $\frac{t-1}{2} - d(c)$ edges to $c$ to create the graph $G'[B]$. Then $\delta(G'[B]) \geq \frac{t-1}{2}$. Let $\gamma=10$; by Claim~\ref{Lemma:Easy Indep}, the graph $G'[B]$ contains either $10(t-1)$ independent edges or $\left\lceil\frac{t-1}{2}\right\rceil$ vertices of degree at least $\frac{n}{21}$.

\setcounter{case}{0}
\begin{case}\label{Case:Matching}
$G[B]$ contains at least $5(t-1)$ independent edges.
\end{case}

This immediately implies $G[B]$ contains $10(t-1-|C|) \geq 5(t-1)$ independent edges.

Let $M = \left\{e_1,\dots,e_{|M|}\right\}$ be a maximum set of independent edges in $G[B]$. Disregard all edges in $M$ whose endpoints are both in $X$; there are at most $\frac{k}{2} < \frac{t-1}{2}$ such edges. 

If $t=1$, then consider an edge $Q = b\beta\in G[B]$, and skip to the construction of $R$. Now suppose $t>1$. 
As noted in the definition of $D'$, every vertex $d'_j\in D'$ contains sufficiently many neighbors in $B'$. As such, for some unique neighbor $\beta'_j$ of each $d'_j$, choose a new matching $M'$ to include every edge $d'_j\beta'_j$ and all other edges in $M$, removing as many as $2|D'|$ edges from $M$ if necessary. Hence, we have $e(M') \geq e(M) - \frac{k}{2} - 2|D'| + |D'| > 2(t-1)$ since $t > 1$. Construct the path $P_D'$. 
Then $P_D' \supset D'\setminus X$ and $|P_D'\cap B| - |P_D'\cap A| = 2|D'\setminus X|-1$ or $|P_D'\cap B| - |P_D'\cap A| = 2|D'\setminus X|$.

Let $M'_X$ be the set of edges in $M'$ that contain a single vertex in $X$ as an endpoint. For all edges $x_i\beta_i\in M'_X$, let $\alpha_i\in A$ be a unique neighbor of $\beta_i$. In a similar fashion to the way we constructed $P_D'$, we may extend $P_D'$ into a path $Q$ by adjoining as many edges in $M' \setminus M'_X$ as necessary so that we have
\begin{equation*}
|Q\cap B| - |Q\cap A| = \tau - |X\cap D'|. 
\end{equation*}
The inequalities $|M'| > 2(t-1) > t \geq \tau$ ensure that $M'$ contains enough edges to extend $P_D'$ into $Q$ (we are still assuming $t > 1$). Construct $Q$ so that the last vertex in $Q$ is in $A$, so that $Q$ begins in $B$ and ends in $A$.

Note that all vertices in $D\setminus D'$ have enough neighbors in $A$ so that each vertex in $D\setminus D'$ has two unique neighbors in $A\setminus X$. For all $x_i\in (D\setminus D')\cap X$, simply consider the unique neighbor $\alpha_i$. We may string all vertices in $D\setminus (D'\cup X)$ into a path $R$ so that $R$ contains as many vertices in $D\setminus (D'\cup X)$ as in $A$. We do this as follows. 
For all $d_j\in D\setminus (D'\cup X)$, choose the unique neighbors $a_j,\alpha_j\in A\setminus X$. For each $j \geq 1$, the vertices $\alpha_j$ and $a_{j+1}$ share a common neighbor $b'_j\in B'$. This gives a path $R = \left\{d_1,a_1,b'_1,\alpha_1,d_2,a_2,b'_2,\alpha_2,\dots,d_{|D\setminus (D'\cup X)|},a_{|D\setminus (D'\cup X)|},b'_{|D\setminus (D'\cup X)|},\alpha_{|D\setminus (D'\cup X)|}\right\}$. 

Hence, the collection of paths
\begin{equation*}
Q' = Q \cup R \cup \left(\bigcup_{x_i\in D'\cap X}\{x_i,\beta_i,\alpha_i\}\right) \cup \left(\bigcup_{x_j\in (D\setminus D')\cap X}\{x_j,\alpha_j\}\right) 
\end{equation*}
satisfies
\begin{equation}\label{Equation:Matching Bound}
|Q'\cap B| - |Q'\cap A| = \tau
\end{equation}
and contains all vertices in $D$. 

\begin{case}\label{Case:Star}
$G[B]$ contains at least $\left\lceil\frac{t-1}{2}\right\rceil$ vertices of degree at least $\frac{n}{21}$. 
\end{case}

This immediately implies $G[B]$ contains $\left\lceil\frac{t-1}{2}\right\rceil$ stars with greater than $\frac{n}{21}-|C| > \frac{n}{22}$ vertices in $G[B] \setminus C$.

Call this collection of stars $\mathcal{S}$, and for each star in $\mathcal{S}$, consider the smaller star $S_j = \{b'_j,c_j,\beta'_j\}$ with center $c_j$ and unique neighbors $b'_j,\beta'_j\in B'\setminus X$, which we may choose due to the large neighborhoods of all $c_j$. Then there exists a unique $a_j$ in the neighborhoods of both $b_j$ and $\beta_j$. Construct $P_D'$; then $P_D' \supseteq D'\setminus X$, and $|P_D' \cap B| - |P_D'\cap A| = 2|D'\setminus X| - 1$. 
Extend $P_D'$ to $Q$ by adjoining as many stars $S_j$ as necessary (and of course the necessary common neighbors in $A$).

\begin{claim}\label{Claim:Enough Stars}
There exist enough stars $S_j$ so that we have
\begin{equation}\label{Equation:Need Q}
|Q\cap B| - |Q\cap A| = \tau - |D'\cap X|
\end{equation}
with $Q$ beginning in $B$ and ending in $A$.
\end{claim}
\begin{proof}
By assumption, we have at least $\left\lceil\frac{t-1}{2}\right\rceil - |D'\cap X|$ disjoint stars $S_j$ in $B\setminus X$. Since each star and an adjacent vertex in $A$ contribute 3 vertices in $B$ and 1 vertex in $A$, we have at least $2\left(\left\lceil\frac{t-1}{2}\right\rceil - |D'\cap X|\right) \geq t - 1 - 2|D'\cap X|$ additional vertices in $B$ as a result of adjoining the stars $S_j$ with a neighbor in $A$.

Recall from Claim~\ref{Lemma:Tau Bounds} that $0 \leq \tau \leq t - 3|D'\cap X| \leq t$. It follows that we need
\begin{equation*}
\begin{aligned}
|Q\cap B| - |B\cap A|   &= \tau - |D'\cap X| \\
			&\leq \big(|B|-|A|+(k-|D'\cap X|) - 0 - |D'\cap X| - 0\big) - |D'\cap X| \\
			&\leq t - 3|D'\cap X|.
\end{aligned}
\end{equation*}
Since we may construct $Q$ to give us as many as $t - 1 - 2|D'\cap X|$ additional vertices in $B$, we are done unless $|D'\cap X| = 0$ and $t$ is odd. (If $t$ is even, then $\left\lceil\frac{t-1}{2}\right\rceil = \frac{t}{2}$, and we have $t - 2|D'\cap X| \geq t - 3|D'\cap X| \geq \tau - |D'\cap X|$ additional vertices in $B$.) If $|D'\cap X| = 0$ and $t$ is odd, then either $n$ is odd and $k$ is even, or $n$ is even and $k$ is odd. Either way, we have $|O_A| \leq k-1$, which implies $\tau - |D'\cap X| \leq |B|-|A|+k-1-2|D'\cap X| < t - 3|D'\cap X|$. Therefore, we can always construct $Q$ to satisfy~\eqref{Equation:Need Q}. Furthermore, we can always have $Q$ end in $A$ by either starting $P_D'$ at $d'_1$ or at $b'_1$. This avoids a possible parity issue, as each star contributes 2 additional vertices in $B$.
\end{proof}

For all $x_i\in D'\cap X$, consider the unique neighbor $\beta_i$, which in turn has the unique neighbor $\alpha_i$. For all vertices $d_j\in D\setminus (D'\cup X)$, let $\alpha_j$ and $a_j$ be the unique neighbors in $A\setminus X$. 

As in Case~\ref{Case:Matching}, create the path $R$ containing all vertices in $D\setminus (D'\cup X)$; naturally, this implies $R$ contains as many vertices in $A$ as in $B$. 
Then the collection 
\begin{equation*}
Q' = Q\cup R \cup \left(\bigcup_{x_i\in D'}\{x_i,\beta_i,\alpha_i\}\right) \cup \left(\bigcup_{x_j\in (D\setminus D')\cap X}\{x_j,\alpha_j\}\right)
\end{equation*}
satisfies
\begin{equation*}
|Q'\cap B| - |Q'\cap A| = \tau
\end{equation*}
and contains all vertices in $D$. 

In both Cases~\ref{Case:Matching} and~\ref{Case:Star}, we have created $Q'$, a collection of paths that contains the necessary number of extra vertices in $B$ than in $A$ and all vertices in $D'$.

We use induction to construct all desired paths $P_1,\dots,P_k$ with desired lengths $n_1,\dots,n_k$. Without loss of generality, assume $n_1 \leq \dots \leq n_k$.

For the base case, let $i=1$. 
Choose two sets $U_1\subset A$ and $V_1\subset B$, each of order $\max\left\{\frac{n}{25k},\left\lceil\frac{n_1}{2}\right\rceil\right\}$, and so that $U_1$ or $V_1$ contains $x_1$ and no other vertices in $Q'\cup X$. By Lemma~\ref{Lemma:Bipanconnected}, the balanced bipartite graph $U_1\cup V_1$ is bipanconnected, and so there exists a path $P_1$ beginning at $x_1$ of length $n_1$ within $U_1\cup V_1$. 

Let $X^i=X\setminus \{x_1,\dots,x_i\}=\{x_{i+1},\dots,x_k\}$, let $m_i=\left\lceil\max\left\{n_i,\frac{n}{25k}\right\}\right\rceil$, and let $\mathscr{P}^i=\bigcup_{j=1}^iP_j$.

Now consider the inductive step for $1 < i < k$. We have 
\begin{equation*}
\begin{aligned}
n-|\mathscr{P}^{i-1}\cup Q\cup X^i| &\geq \frac{(k-i+1)n}{k} - (k-i+1) - \frac{n}{25k} \\
				      &> \frac{n}{k}.
\end{aligned}
\end{equation*}
Hence, there are always enough vertices to create a balanced bipartite graph $U_i\cup V_i\subset (A\cup B)\setminus (\mathscr{P}^{i-1}\cup Q\cup X^i)$ that includes $x_i$ with $m_i$ vertices. By Lemma~\ref{Lemma:Bipanconnected}, we may construct a path $P_i$ in $U_i\cup V_i$ beginning at $x_i$ with $n_i$ vertices.

Finally, for $i=k$, we wish for $P_k$ to contain all vertices in $V(G)\setminus \mathscr{P}^{k-1}$. Note that this final path must include $Q$ and $R$. We have $X^k=\emptyset$, and hence
\begin{equation*}
\begin{aligned}
n - |\mathscr{P}^{k-1}\cup Q\cup R\cup X^k|  &> n - \left(\frac{n}{25k} + \frac{k-1}{k}\right)n \\
					&> \frac{24n}{25k}.
\end{aligned}
\end{equation*}
By the definition of $Q$ and $R$, the graph $G\setminus \bigcup_{i=1}^{k-1}P_i$ is within a single vertex of being a balanced bipartite graph, and $|Q| + |R| \leq 4\tau + 2|D| \ll \frac{n}{25k}$. Also note that $Q$ and $R$ both start in $A$ and end in $B$; hence, we may adjoin $Q$ and $R$ using Lemma~\ref{Lemma:Bipanconnected} and some vertices $\{a_q,b_r\}$. Similarly, the amalgamation $Q\cup \{a_q,b_r\}\cup R$ still begins in $B$ and ends in $A$, allowing us to perform a similar process in adjoining $\{x_k,(\beta_k),\alpha_k\}$ (if $\beta_k$ exists) to $Q\cup \{a_q,b_r\}\cup R$. Lastly, use Lemma~\ref{Lemma:Bipanconnected} to adjoin $Q\cup \{a_q,b_r\}\cup R$ to a remaining vertex in $G - \mathscr{P}^{k-1}$, possibly adding on one last vertex if $V(G - \mathscr{P}^{k-1})$ is odd. 
This results in a path $P_k$ that begins at $x_k$ and necessarily contains $n_k$ vertices.

We have therefore created disjoint paths beginning at $x_i$ on $n_i$ vertices for all $1\leq i \leq k$ that together cover $V(G)$.
\end{proof}

\section{Proof of Lemma~\ref{Lemma:Connectivity}}\label{Section:Connectivity}

We begin with a lemma ensuring that low connectivity in the reduced graph $R$ results in at most two components in the original graph $G$.

\begin{lemma}\label{Lemma:G=AB}
Let $\epsilon,\d > 0$ be small reals and $k$ be a positive integer. If $G$ is a graph with $\sigma_2(G) \geq n+k-1$ and reduced graph $R$ with connectivity at most $\left(\frac{1}{10}-\frac{3}{5}(\d+2\epsilon)\right)|R|$, then $R$ consists of only two components {\rm(}and a cutset if $\kappa(R) > 0${\rm)}.
\end{lemma}
\begin{proof}
Applying Lemma~\ref{Lemma:Regularity} to $G$, let $G''=G'[V(G)\setminus V_0]$. Since $d_{G''}(v) > d_G(v) - (\d+2\epsilon)n$, it immediately follows that $\sigma_2(R) > (1 - 2(\d+2\epsilon))|R|$. Let $D$ be a cutset of $R$ (if one exists). Suppose $R$ (or $R\setminus D$) contains at least 3 components, three of which being $A$, $B$, and $C$. Let $a\in A$, $b\in B$ and $c\in C$. Then $d(a)+d(b) > (1-2(\d+2\epsilon))|R|$, which implies $|A|+|B| > (1-2(\delta+2\epsilon))|R| - 2|D|$. Similarly, the same is true for $|B|+|C|$ and $|A|+|C|$. So $2(|A|+|B|+|C|) > 3(1-2(\d+2\epsilon))|R| - 6|D|$, or $|D| > \left(\frac{1}{10}-\frac{3}{5}(\d+2\epsilon)\right)|R|$, a contradiction.
\end{proof}

Note that the connectivity of $R$ may be considerably larger than $\kr|R|$. Also note that $A$, $B$, $C$, and $D$ were vertex sets in $R$. In the following remark, the same symbols are used to denote vertex sets in $G$.

\begin{remark}\label{Remark:Split G}
Given small real numbers $\epsilon, \d > 0$ and a positive integer $k$, let $G$ be a graph of order $n = \sum_{i=1}^k n_i \geq n(\epsilon,\d,k)$ with $\sigma_2(G) \geq n+k-1$ and $\delta(G) \geq \frac{n_k}{8}$. If the reduced graph of $G$ has connectivity at most $\kr|R|$, then let $D\subset V(G)$ be the cluster corresponding to a cut vertex of $R$. (If $R$ contains no cut vertices, then $D=\emptyset$.) Let $V_0$ be the garbage cluster of $G$ resulting from Lemma~\ref{Lemma:Regularity}, and let $C$ be a minimum cutset of $G$. Then $C\subseteq D\cup V_0$. By Lemma~\ref{Lemma:Regularity}, each vertex of $R$ corresponds to a cluster in $G$ of order $L = \xi n$. Hence, we have $k+1 \leq |C| \leq |D| + |V_0| \leq \epsilon \xi |R| n + \epsilon n$. By Lemma~\ref{Lemma:G=AB}, we may define $A$ and $B$ to be the components of $G\setminus C$ and write $G=A\cup C\cup B$. It immediately follows from $\sigma_2(G)\geq n+k-1$ that
\begin{equation}\label{Equation:G=ACB}
\begin{aligned}
\delta(G[A]) &> |A|-|C| > |A| - (\epsilon \xi |R| + \epsilon)n, \\ 
\delta(G[B]) &> |B|-|C| > |B| - (\epsilon \xi |R| + \epsilon)n.    
\end{aligned}
\end{equation}
From the condition $\delta(G) \geq \frac{n_k}{8} \geq \frac{n}{2k}$, we know $|A|,|B| \geq \frac{n_k}{8} - |C| \geq \left(\frac{1}{8k}- \epsilon \xi |R| -\epsilon\right)n > \frac{n}{8(k+1)}$.
\end{remark}

Note that $\epsilon\xi|R| \ll 1$. 

While panconnected sets give paths of arbitrary length, only the endpoints are specified. Hence, to create disjoint paths of arbitrary length, we must create sets using vertices that are not part of an already existing desired path. Fortunately, even small subsets of $A$ and $B$ induce panconnected graphs. 

\begin{lemma}\label{Lemma:Panconnected Sets in B}
Let $\epsilon$, $\d$, $k$, and $G=A\cup C\cup B$ be defined as in Remark~\ref{Remark:Split G}. Then the induced graph on any subgraph of $A$ or $B$ of order at least $2(\epsilon \xi |R|+\epsilon)n$ is panconnected.
\end{lemma}
\begin{proof}
We see from~\eqref{Equation:G=ACB} that $\delta(G[A]) > |A|-|C| > |A| - (\epsilon \xi |R| + \epsilon)n$. 
Then for all $U\subset A$ of order at least $2(\epsilon \xi |R| +\epsilon)n$, we have
\begin{equation*}
\begin{aligned}
\delta(G[U])  &\geq  |S|-(\epsilon \xi |R| + \epsilon)n+1 \\
		     &\geq  \frac{|S|+2}{2}.
\end{aligned}
\end{equation*}
By Theorem~\ref{Thm:Williamson}, the graph $G[U]$ is panconnected. A symmetric argument shows that if $U\subset B$ has order at least $(\epsilon \xi |R| + \epsilon)n$, then $G[U]$ is panconnected.
\end{proof}

With this information, we prove the following lemma which is a bit stronger than Lemma~\ref{Lemma:Connectivity}.

\begin{lemma}
Given small real numbers $\epsilon, \d > 0$ and a positive integer $k$, let $G$ be a graph of order $n = \sum_{i=1}^k n_i \geq n(\epsilon,\d,k)$ with $\sigma_2(G) \geq n+k-1$ and $\delta(G) \geq \frac{n_k}{8}$. If $\kappa(R) \leq \epsilon|R|$, then the conclusion of Conjecture~\ref{Conj:E-O} holds. 
\end{lemma}

\begin{proof} 
Suppose $\kappa(R) \leq \kr|R|$, and let $G=A\cup C\cup B$ as in Remark~\ref{Remark:Split G}.
As noted before~\eqref{Equation:G=ACB}, we know $k+1 \leq |C| \leq (\epsilon \xi |R| + \epsilon)n$. As noted after~\eqref{Equation:G=ACB}, we know $|A|,|B| > \frac{n}{8(k+1)}$. For each $c\in C$, we may reserve 2 unique neighbors $a_c\in A$ and $b_c\in B$. 
Call $A_C = \{a_c\in A\setminus X \;|\; c\in C\}$ (symmetrically $B_C = \{b_c\in B\setminus X \;|\; c\in C\}$) the set of \emph{proxy} vertices in $A$ (symmetrically $B$). Then we have 
\begin{equation*}
|C| = |A_C| = |B_C|.
\end{equation*} 
Given a vertex $x$, let an \emph{$x$-path} be a path containing $x$ as an endpoint. Namely, each desired path $P_i$ in $G$ is an $x_i$-path. 
%
Recall that $X = \{x_i \;|\; n_i\leq n_{i+1}\}$. For some $i\leq k$, let
\begin{equation*}
A^*=A\setminus \bigcup_{j=1}^{i-1}P_j,
\end{equation*}
and define $B^*$ and $C^*$ similarly. In particular, note that $A^*=A$ for $i=1$. Let
\begin{equation*}
\begin{aligned}
A^v = (A^*\setminus (A_C\cup X))\cup v, \\
B^v = (B^*\setminus (B_C\cup X))\cup v;
\end{aligned}
\end{equation*}
i.e., if $v=x_i$ or $v=a_c$, then $A^v\cap (X\cup A_C) = v$, and symmetrically for $B^v$ and $B_C$. For the sake of notation, if $v=x_i$ for some $i$, then we write $A^x$ and $B^x$. This will never cause an issue, as we never discuss $A^*$ or $x_i$ for different values of $i$ at the same time.

We induct on $i$ to prove our result. Consider the base case $i=1$. If $x_1\in A$ and $n_1 \leq |A^x| - 4(\epsilon \xi |R| + \epsilon)n$, then use Lemma~\ref{Lemma:Panconnected Sets in B} to construct an $x_1$-path $P_1\subset A^x$ containing $n_1$ vertices. If $x_1\in A$ and $n_1 > |A^x| - 4(\epsilon \xi |R| + \epsilon)n$, then let $c\in C$ with proxy vertices $a\in A_C$ and $b\in B_C$. Use Lemma~\ref{Lemma:Panconnected Sets in B} to create an $x_1,a$-path $P_A$ consisting of all but $4(\epsilon \xi |R| + \epsilon)n$ vertices of $x_1 \cup A^a$. Also create a $b$-path $P_B\subseteq B^b$ with $n_1 - |P_A| - 1$ vertices. Then $P_1 = P_A\cup c\cup P_B$ is an $x_1$-path with $n_1$ vertices. If $x_1\in B$, then a symmetric argument works. Lastly, if $x_1\in C$, then suppose without loss of generality that $|A\setminus (A_C\cup X)| \geq |B\setminus (B_C\cup X)|$. Since $|C|+|A_C\cup X| \leq 2(\epsilon \xi |R| + \epsilon)n + k < 4(\epsilon \xi |R| + \epsilon)n$ and $n_1\leq \frac{n}{k}$, this implies $n_1 < |A^x\cup a| - 4(\epsilon \xi |R| + \epsilon)n$. Let $a\in A_C$ be the proxy vertex of $x_1$, and use Lemma~\ref{Lemma:Panconnected Sets in B} to create an $a$-path $P_A\subset A^a$ with $n_1 - 1$ vertices. Then $P_1 = x_1\cup P_A$ is an $x_1$-path with $n_1$ vertices.

Now suppose $1 < i < k$ and that the disjoint $x_j$-paths $P_1,\dots,P_{i-1}$ have been constructed in $G$. If $x_i\in A^*\cup C^*$ and $n_i \leq |A^x| - 4(\epsilon \xi |R| + \epsilon)n$, then use Lemma~\ref{Lemma:Panconnected Sets in B} to construct an $x_i$-path $P_i\subset A^x$ containing $n_i$ vertices. If $x_i\in A^*\cup C^*$ and $n_i > |A^x| - 4(\epsilon \xi |R| + \epsilon)n$, then
\begin{equation}\label{Equation:B Big}
\begin{aligned}
|B^b| &=    n - |A^*| - |C^*| - |\{x_i,\dots,x_k\}| - \left|\bigcup_{j=1}^{i-1}P_j\right| \\
      &\geq n - n_i - 4(\epsilon \xi |R| + \epsilon)n - (\epsilon \xi |R| + \epsilon)n - (k-i+1) - \sum_{j=1}^{i-1}n_j \\
	&\geq  \sum_{j=i+1}^{k}n_j - 6(\epsilon \xi |R| + \epsilon)n.
\end{aligned}
\end{equation}
Hence, if $i<k-1$, then $|B^b| > n_{i+1} \geq n_i$. If $i=k-1$, then by the Pigeonhole Principle, we have
\begin{equation}\label{Equation:B Big Set}
\begin{aligned}
|B^b|   &> n_k-6(\epsilon \xi |R| + \epsilon)n \\
	&> \frac{n}{k} - 6(\epsilon \xi |R| + \epsilon)n \\
\end{aligned}
\end{equation}
Furthermore, since $|C^*\setminus x_i| \geq k+1 - (i-1) - 1\geq 2$, we know there exists $c\in C^*$ with proxy vertices $a\in A_C$ and $b\in B_C$. If $|A^x| > 4(\epsilon \xi |R| + \epsilon)n$, then use Lemma~\ref{Lemma:Panconnected Sets in B} to create an $x_i,a$-path $P_A$ consisting of all but $4(\epsilon \xi |R| + \epsilon)n$ vertices of $a\cup A^x$. If $|A^x|\leq 4(\epsilon \xi |R| + \epsilon)n$ and $a\in A^*$, then use Lemma~\ref{Lemma:Panconnected Sets in B} to create an $x_i,a$-path $P_A$ consisting of three vertices in $a\cup A^x$. If $|A^x|\leq 4(\epsilon \xi |R| + \epsilon)n$ and $a\in C^*$, then let $P_A=\emptyset$. Regardless of the initial size of $A^x$, we now have
\begin{equation}\label{Equation:Ax big enough}
3(\epsilon \xi |R| + \epsilon)n < |A^x \setminus P_A| \leq 4(\epsilon \xi |R| + \epsilon)n.
\end{equation}
From~\eqref{Equation:B Big}~and~\eqref{Equation:B Big Set}, we may similarly use Lemma~\ref{Lemma:Panconnected Sets in B} to create a $b$-path $P_B\subset B^b$ with $n_i - |P_A| - 1$ vertices. Then $P_i = P_A\cup c\cup P_B$ is an $x_i$-path with $n_i$ vertices. 
If $x_i\in B^*$, then a symmetric argument works.

Finally, suppose $i=k$ and that the disjoint $x_j$-paths $P_1,\dots,P_{k-1}$ have been constructed in $G$. 
From~\eqref{Equation:B Big Set}~and~\eqref{Equation:Ax big enough}, we know
\begin{equation}\label{Equation:Make}
\begin{aligned}
|A^* \setminus(X\cup A_C)| > 3(\epsilon \xi |R| + \epsilon)n, \\
|B^* \setminus(X\cup B_C)| > 3(\epsilon \xi |R| + \epsilon)n.
\end{aligned}
\end{equation}
From the Pigeonhole Principle, we also know
\begin{equation*}
n_k \geq \frac{n}{k} \gg |A_C| + |C^*| + |B_C|.
\end{equation*}
Without loss of generality, assume $x_k\in A^*\cup C^*$. Note that $|C^*\setminus x_k| \geq 1$, and hence, that we must have $n_k > |A^x|$. Define $A_C^*$ and $B_C^*$ similarly to the way $A^*$ and $B^*$ are defined. Use Lemma~\ref{Lemma:Panconnected Sets in B} repeatedly $|C^*| \leq (\epsilon \xi |R| + \epsilon)n$ times within $A^*$ and within $B^*$ each to create a path $P_C$ that strings together all vertices in $C^*$ by using all proxy vertices in $A_C^*$ and $B_C^*$. Since $A_C^*$, $B_C^*$, and $C^*$ each have at most $(\epsilon \xi |R|+\epsilon)n$ vertices, and since $|C^*\setminus x_k| \geq 1$, we know $7 \leq |P_C| < 5(\epsilon \xi |R| + \epsilon)n$. Given the high value of $\delta(G)$, we may ensure that $P_C$ starts with a proxy vertex in $A^*$ and ends with a proxy vertex in $B^*$ by including an additional vertex in $A^*$ or $B^*$ adjacent to some vertex in $C^*$. Let the endpoints of $P_C$ be $a\in A^*$ and $b\in B^*$. Noting~\eqref{Equation:Make}, we may use Lemma~\ref{Lemma:Panconnected Sets in B} in $A^*$ and $B^*$ to create an $x_k,a$-path $P_A$ that contains all vertices in $A^*\setminus P_C$. Similarly, use~\eqref{Equation:Make} and Lemma~\ref{Lemma:Panconnected Sets in B} to create a $b$-path $P_B$ that contains all vertices in $B^*\setminus P_C$. Then $P_k = P_A\cup P_C\cup P_B$ is an $x_k$-path that contains all remaining $n_k$ vertices in $G$.

We have created $k$ paths $P_1,\dots,P_k$ in $G$, with each path $P_i$ starting at $x_i$ and having $n_i$ vertices.
\end{proof}


\bibliography{ref}
\bibliographystyle{plain}

\end{document}